\documentclass{article}
\usepackage{amssymb,stmaryrd,nicefrac,enumitem,bm,stackrel,tikz}
\usetikzlibrary{patterns}
\usepackage{palatino}
\usepackage{url}
\usepackage[hidelinks]{hyperref}
\usepackage[utf8]{inputenc}
\usepackage[small]{caption}
\usepackage{graphicx}
\usepackage{amsmath}
\usepackage{amsthm}
\usepackage{booktabs}
\usepackage{algorithm}
\usepackage{algorithmic}
\newcommand{\val}[1]{\llbracket #1 \rrbracket}
\newcommand{\neqrel}{\gg}
\newcommand{\varkur}{{\rm Kur}_{\boxdot \neq }}

\newtheorem{theorem}{Theorem}[section]

\theoremstyle{definition}
\newtheorem{definition}[theorem]{Definition}

\newtheorem{lemma}[theorem]{Lemma}

\newtheorem{fun-fact}[theorem]{Fun-Fact}

\newtheorem{remark}[theorem]{Remark}

\title{Taming the `elsewhere':\\
On expressivity of topological languages}

\author{%
David Fern\'andez-Duque\\
Department of Mathematics WE16, Ghent University\\
{\tt   David.FernandezDuque@UGent.be}
}

\begin{document}

\maketitle

\begin{abstract}
In topological modal logic, it is well known that the Cantor derivative is more expressive than the topological closure, and the `elsewhere,' or `difference,' operator is more expressive than the `somewhere' operator.
In 2014, Kudinov and Shehtman asked whether the combination of closure and elsewhere becomes strictly more expressive when adding the Cantor derivative.
In this paper we give an affirmative answer: in fact, the Cantor derivative alone can define properties of topological spaces not expressible with closure and elsewhere.
To prove this, we develop a novel theory of morphisms which preserve formulas with the elsewhere operator.
\end{abstract}

\section{Introduction}

{\em Topology} can be described as the qualitative study of space, and as such it is not surprising that it often serves as the foundation for spatial reasoning.
One way to think about topological spaces is as pairs $(X,{\bm i})$ consisting of a set $X$ and an operator ${\bm i} \colon \wp(X)\to \wp(X)$ satisfying ${\bm i} X = X$, ${\bm i}(A\cap B) = {\bm i}A \cap {\bm i}B$, and ${\bm i}  A \subseteq A \cap {\bm i}{\bm i} A $.
The set ${\bm i} A$ is the {\em interior} of $A$, and the intuition is that it should not contain points on the {\em boundary} of $A$: if we think of $A$ as an orange, we can think of ${\bm i}A$ as the same orange without its peel.
If $A = {\bm i}A$, we say that $A$ is {\em open;} open sets are those that do not contain any of their boundary points.
The interior operator admits a dual `closure' operator, denoted, $\bm c$, defined by ${\bm c}A = X\setminus {\bm i}(X\setminus A)$; if $A$ is now the peeled orange, ${\bm c}A$ would be the original orange with its peel.

McKinsey and Tarski~\cite{mctarski} already observed that the topological closure and interior could be used to provide semantics for modal logic.
In this setting, a proposition $\varphi$ is interpreted as a region $\val\varphi\subseteq X$, and $\boxdot\varphi$ is interpreted as its interior ${\bm i}\val\varphi$ (see Section \ref{SecSem} for formal definitions).
If we moreover add a universal modality $\forall$, where $\forall\varphi$ is true iff $\varphi$ covers the entire space $X$ \cite{ShehtmanEverywhereHere}, then we obtain a modal framework in which all spatial relations of $\sf RCC8$ may be expressed \cite{Wolter00}: for example, $\forall(p\to \boxdot q)$ states that the region $\val p$ is non-tangentially contained within the region $\val q$.

However, the closure and universal modality are not the only primitive operations one can use for this style of spatial reasoning.
McKinsey and Tarski~\cite{mctarski} also noted that the Cantor derivative gives rise to an alternative interpretation of modal logic.
For $A\subseteq X$, we define ${\bm d}A$ to be the set of points $x\in X$ such that $x\in {\bm c}(A\setminus \{x\})$; ${\bm d}A$ is also called the set of {\em limit points} of $A$.
Dually, we may define $x\in {\bm p}A $ if $x\in {\bm i}(A\cup\{x\}) $: this is the set of points that have a {\em punctured neighborhood} contained in $A$.
Modal logic based on the Cantor derivative does not validate the reflexivity property $\Box p\to p$, making it an attractive model of {\em belief,} rather than knowledge.
Logics of the Cantor derivative have been studied extensively (see e.g.~\cite{beg,beg2010,Lucero-Bryan11}), particularly in the context of {\em scattered spaces,} which have applications to the logic of provability \cite{abashidze1985,AbaBlass,blass1990}.

The universal modality also comes with a `punctured' variant in the {\em elsewhere,} or {\em difference,} modality $[\neq]$ \cite{Gargov1993}, studied in a topological context in e.g.~\cite{Gabelaia2001,Kudinov2006}.
Here, $[\neq]\varphi$ holds on $x$ if $\varphi$ is true in every point, aside from possibly $x$.
There are some compelling reasons for considering the punctured variants as primitive, rather than the `unpunctured' ones: for example, we may easily define $\forall\varphi \equiv \varphi\wedge [\neq]\varphi$.
In fact, the punctured variants are {\em strictly more expressive.}

To make this precise, note that each $M\subseteq \{\Box,\boxdot,[\neq],\forall\}$ gives rise to a propositional modal language $\mathcal L_M $ (where we omit set-brackets and also brackets around $\neq$; see Section \ref{SecSem}).
For example, $\mathcal L_{\Box\neq}$ denotes the language with modalities $\Box$ and $[\neq]$; all modalities are definable in terms of these, so it is rarely useful to consider more than two modalities at once.
One question that arises when desigining a language for formal topological reasoning is how the different combinations of modalities compare with respect to expressive strength.
A general analysis of the situation is given in \cite{KS} (see Section \ref{SecRel}).

In particular, it is known (and follows from the above discussion) that $\mathcal L_{\boxdot\neq}$ is reducible to $\mathcal L_{\Box \neq }$, in the sense that any class of topologies definable in the former is already definable in the latter (see Section \ref{SecRel}).
Kudinov and Shehtman ask whether the reduction is {\em strict,} in the sense that converse fails.
The aim of this paper is to give an affirmative answer to the latter.
In fact, we show that already $\mathcal L_\Box$ cannot be reduced to $\mathcal L_{\boxdot\neq}$.

We prove this by considering a variant of the {\em local $1$-componency} property, which states that if $x\in U$ and $U$ is open, then there is a neighborhood $N\subseteq U$ of $x$ such that $N\setminus \{x\}$ is connected.
This property holds on e.g.~$\mathbb R^2$ but not on $\mathbb R$, and it implies the validity of the formula ${\rm Kur}:=\Box(\boxdot p\vee \boxdot \neg p)\to \Box p \vee \Box \neg p$.
In Section \ref{secDefKur}, we show that ${\rm Kur}$ itself is definable in $\mathcal L_{\boxdot\neq}$, in the sense that there is a formula $\varkur $ of this language such that for any topological space $X$, $X\models{\rm Kur}$ if and only if $X\models\varkur$.
However, in Section \ref{sectIndef} we also show that a mild variant of this property, $\Box{\rm Kur}$, is not definable in terms of $\boxdot$ and $[\neq]$.
Intuitively, $\Box{\rm Kur}$ states that local $1$-componency may only fail on a discrete set of points.
To prove that this formula is not expressible in $\mathcal L_{\boxdot\neq}$, in Section \ref{secMorph} we first introduce a new notion of morphism for the difference modality.

\section{Topological spaces}

In this section we briefly recall some background from topology, particularly the notion of Cantor derivative on a topological space.

\begin{definition} (\emph{topological space})
A topological space is a pair $(X,\tau)$, where $X$ is a set and $\tau$ is a subset of $\wp(X)$ that satisfies the following conditions:
\begin{itemize}
\item $X,\varnothing\in\tau$;
\item if $U,V\in\tau$, then $U\cap V\in\tau$;
\item if $\mathcal{U}\subseteq \tau$, then $\bigcup\mathcal{U}\in\tau$.
\end{itemize}
The elements of $\tau$ are called {\em open sets,} and the complement of an open set is called a \emph{closed} set. 
\end{definition}

We will notationally identify $X$ with $(X,\tau)$, which in this paper will not lead to ambiguity as all topologies we consider come from the Euclidean spaces $\mathbb R^n$ or their subspaces.
Perhaps the most familiar example of a topological space is the real line $\mathbb R$, where $U\subseteq \mathbb R$ is open iff it is a (possibly infinite) union of intervals of the form $(a,b)$.
More generally, each space $\mathbb R^n$ comes with a topology where $U$ is open iff whenever $x\in U$, there is $\varepsilon>0$ such that $\|x-y\|<\varepsilon$ implies that $y \in U$ (where $\|\cdot\|$ is the standard Euclidean norm).
Each space $\mathbb R^n$ is connected: recall that if $X$ is a topological space and $C\subseteq X$, we say that $C$ is {\em connected} if whenever $C\subseteq A\cup B$ with $A,B$ disjoint and open, it follows that $C\cap A =\varnothing$ or $C\cap B = \varnothing$.

If $X,Y$ are topological spaces, a function $f\colon X\to Y$ is {\em continuous} if $f^{-1}(B)$ is open whenever $B\subseteq Y$ is open, and {\em open } if $f (A)$ is open whenever $A\subseteq X$ is open.
A continuous and open map is an {\em interior map.}

The fundamental operation on topological spaces we are interested in is the {\em Cantor derivative.}

\begin{definition}(\emph{Cantor derivative}) Let $ (X,\tau)$ be a topological space. Given $S\subseteq X$, the \emph{Cantor derivative} ${\bm d} $ of $S$ is the set ${\bm d} S $ of all limit points of $S$, i.e. 
$ x\in  {\bm d} S $ if and only if, whenever $x\in U\in \tau $, it follows that $ (U\cap S)\backslash\{x\}\neq\varnothing$.
\end{definition}

When working with more than one topological space, we may denote the Cantor derivative of the topological space $X$ by ${\bm d}_X $.
Given $A,B\subseteq X$, it is not hard to check that the Cantor derivative satisfies ${\bm d} \varnothing =\varnothing$, ${\bm d}(A\cup B)= {\bm d} A \cup {\bm d} B $, and ${\bm d}{\bm d} A  \subseteq A\cup {\bm d} A $.

The {\em topological closure} of $A\subseteq X$ can then be defined as ${\bm c}A = A \cup {\bm d}A$, or, directly, as the intersection of all closed sets containing $A$.
Both the Cantor derivative and topological closures admit duals, the {\em punctured interior} given by ${\bm p}A = X\setminus {\bm d}(X\setminus A)$ and the {\em interior} ${\bm i}A = X\setminus {\bm c}(X\setminus A)$.
The latter is also definable as the union of all open sets contained in $A$, while the former has the property that $x\in {\bm p}A$ iff there is some open set $U$ such that $x\in U\subseteq A\cup \{x\}$.
We say that $A$ contains a {\em punctured neighborhood} of $x$.

\section{Topological languages}\label{SecSem}

All languages we consider will be subsets of the full topological language $\mathcal L_*$ given by the following syntax in Backus-Naur form:
\[\varphi,\psi :: =   \ p \  \mid \ \neg \varphi \ \mid \ \varphi\wedge\psi \ \mid \ \Box \varphi\ \mid \ \boxdot \varphi \ \mid  \
[{\neq}]\varphi\ \mid \ \forall \varphi\]
Sublanguages are indicated with allowed modalities as subindices. We are mostly concerned with $\mathcal L_\Box$ and $\mathcal L_{\boxdot{\neq}}$ (we omit brackets around $\neq$).
As usual, we use $\Diamond$ as a shorthand for $\neg\Box\neg$ and $\langle \neq\rangle$ as a shorthand for $\neg[\neq]\neg$.

\begin{definition}\label{d-semantics} 
A \emph{topological model} is a pair $\mathcal {M}=(X,\val\cdot)$ where $X$ is a topological space and $\val\cdot\colon\mathcal L_* \rightarrow \wp(X)$ is a valuation function assigning a subset of $X$ to each formula of $\mathcal L_*$, satisfying the following recursive clauses:  
\begin{itemize}

	\item $\val{\neg \varphi} = X \setminus \val\varphi $,
	\item $\val{ \varphi\wedge\psi } = \val \varphi \cap \val \psi $,
	\item $\val{\Box\varphi} = {\bm p}\val\varphi $,
		\item $\val{\boxdot\varphi} = {\bm i}\val\varphi $,
		\item $\val{[\neq]\varphi} = \{x\in X: \val\varphi \cup \{x\} =X\}$, and
\item $\val{\forall\varphi} = X$ if $\val\varphi =X$, otherwise $\val{\forall \varphi} =\varnothing$.
\end{itemize}
\end{definition}

Note that, dually, $x\in \val{\langle\neq\rangle\varphi}$ if and only if there is $x'\neq x$ such that $x'\in \val\varphi$.
We write $(\mathfrak  M,x)\models\varphi$ if $x\in \val\varphi$, and $\mathfrak  M\models\varphi$ if $\val\varphi = X$.
Similarly, $X\models\varphi$ if $(X,\val\cdot)\models \varphi$ for every valuation $\val\cdot$ on $X$, and $(X,x)\models\varphi$ if $x\in \val \varphi$ for every valuation $\val\cdot$.
We may write $\val\cdot_X$ instead of $\val\cdot$ when working with more than one topological space.

\section{Reducibility between languages}\label{SecRel}

The general type of question we are concerned with is: {\em Given $\mathcal L,\mathcal L'\subseteq \mathcal L_*$, is $\mathcal L'$ reducible to $\mathcal L $?}
Intuitively, $\mathcal L'$ is reducible to $\mathcal L $ if every $\mathcal L'$-definable class of spaces is also $\mathcal L$-definable.
Let us make this precise.

\begin{definition}
Given $\varphi\in \mathcal L_*$, let $\mathcal C(\varphi)$ be the class of all topological spaces $X$ such that $X\models\varphi$.
Say that a class of spaces $\Omega$ is {\em definable} in $\mathcal L \subseteq \mathcal L_*$ if there is $\varphi\in\mathcal L$ such that $\Omega = \mathcal C(\varphi)$, and similarly, say that a formula $\psi$ is definable in $\mathcal L$ if there is $\psi'\in\mathcal L$ such that $\mathcal C(\psi') = \mathcal C(\psi )$.

Then, write $\mathcal L'\leq\mathcal L $ if every $\varphi\in \mathcal L'$ is definable in $\mathcal L $, and $\mathcal L'<\mathcal L$ if $\mathcal L'\leq \mathcal L$ but $\mathcal L\not\leq\mathcal L'$.
In this case, we say that $\mathcal L'$ is {\em (strictly) reducible to $\mathcal L$.}
\end{definition}

Kudinov and Shehtman~\cite{KS} give an overview of the known facts about reducibility between sublanguages of $\mathcal L'$.
The general picture is represented in Figure \ref{figLanguages}.
The non-strict inclusions all follow from the definability of $\boxdot$ in terms of $\Box$ and $\forall$ in terms of $[\neq]$, which we have already discussed.
Languages obtained by adding $\forall$ to $\Box$ or $\boxdot$ are more expressive because e.g.~$\forall$ is needed to define connectedness \cite{ShehtmanEverywhereHere}, and languages obtained by replacing $\forall$ with $[\neq]$ are more expressive because $[\neq]$ is needed to distinguish between e.g.~the real line and the circle \cite{KS}.
That no language with $\Box$ is reducible to one without $\Box$ follows from results in the current paper, but was already known for languages without $\neq$ since $\mathcal L_\boxdot$ cannot distinguish between $\mathbb R$ and $\mathbb R^2$ \cite{mctarski} but $\mathcal L_\Box$ can \cite{KS,Lucero-Bryan13}.

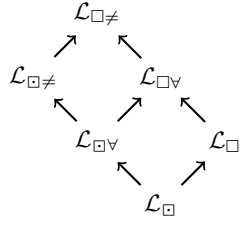
\begin{figure}

\begin{center}

	\begin{tikzpicture}[thick,->,auto,font=\small,node distance=1.2cm]
	\node[] (logbasic) {$\mathcal L_{\boxdot}$};
	\node[] (logrnax) 	[above left of=logbasic] {$\mathcal L_{\boxdot\forall}$};	
	\node[] (logexp) [above left of=logrnax] {$\mathcal L_{\boxdot\neq}$};
	\node[] (logrnfs)  [above right of=logrnax] {$\mathcal L_{\Box\forall}$};		
	\node[] (loghomeo)  [below right of=logrnfs] {$\mathcal L_\Box$};	
	\node[] (logpers) [above left of=logrnfs] {$\mathcal L_{\Box\neq}$};
	
	\path[] 
	(logbasic) edge[] node[pos=0.5,right,font=\scriptsize]{}(loghomeo)
	(logbasic) edge[] node[pos=0.5,left,font=\scriptsize]{}(logrnax)	
	(loghomeo) edge[] node[pos=0.5,right,font=\scriptsize]{}(logrnfs)		
	(logexp) edge[] node[pos=0.5,right,font=\scriptsize]{}(logpers)	
	(logrnax) edge[] node[right,right,pos=0.5,font=\scriptsize]{}(logexp)
	(logrnax) edge[] node[right,pos=0.5,font=\scriptsize]{}(logrnfs)		
	(logrnfs) edge[] node[right,pos=0.5,font=\scriptsize]{}(logpers)	
	;								
	\end{tikzpicture}

\end{center}	
	\caption{Order relations between different spatial languages: an arrow $\mathcal L\to\mathcal L'$ indicates that $\mathcal L < \mathcal L'$.
	Our main result is that $\mathcal L_{\Box}\not\leq\mathcal L_{\boxdot\neq}$, thus establishing that the diagram is complete: all reductions are already indicated in the figure.}
	\label{figLanguages}
\end{figure}

Specifically, the $\mathcal L_\Box$-logics of $\mathbb R$ and $\mathbb R^2$ differ because $\mathcal L_\Box$ contains a formula valid on {\em locally $1$-component} spaces.
If $X$ is a topological space, $x\in X$ is locally $1$-component if for every neighborhood $U$ of $x$ there is a neighborhood $N \subseteq U$ of $x$ such that $N\setminus \{x\}$ is connected.
The space $X$ is locally $1$-component if every point of $X$ is locally $1$-component.
It is well known and easy to see that $\mathbb R^2$ is locally $1$-component, but $\mathbb R$ is not.
(In fact, the distinction between the circle and the line noted by \cite{KS} is a global version of this property: after removing a point, the circle remains connected, but the line does not.)

It is well known that locally $1$-component spaces validate the Kuratowski formula
\[{\rm Kur}:= \Box(\boxdot p\vee\boxdot \neg p) \to \Box p \vee\Box \neg p\]
\cite{Kuratowski22}, but since this is an important element in our own results, we provide a proof.

\begin{lemma}\label{lemKurAx}
Let $X$ be any topological space and $x\in X$.
Then, if $x$ satisfies local $1$-componency, it follows that $(X,x) \models \rm Kur$.
\end{lemma}

\begin{proof}
Assume that $x$ satisfies local $1$-componency.
To show that $(X,x) \models \rm Kur$, assume moreover that $\val\cdot$ is a valuation such that $x\in \val{\Box(\boxdot p\vee\boxdot \neg p)}$.
Then, $x$ has a neighborhood $U$ such that $U\setminus \{x\}\subseteq \val{ \boxdot p\vee\boxdot \neg p }$; by local $1$-componency, we may assume that $U\setminus \{x\}$ is connected (otherwise, choose a suitable $U'\subseteq U$).
Then, $U \setminus \{x\} \subseteq \val{\boxdot p} \cup  \val{\boxdot \neg p}$.
Since these two sets are open and disjoint and $U\setminus \{x\}$ is connected, we either have that $U \setminus \{x\} \subseteq \val{\boxdot p}$ and $x\in\val{\Box p}$, or $U \setminus \{x\} \subseteq \val{\boxdot \neg p}$ and $x\in\val{\Box \neg p}$. 
Either way, $x\in \val{\Box p\vee\Box \neg p}$.
\end{proof}

Kudinov and Shehtman~\cite{KS} point out that $\mathcal L_{\boxdot\neq} \leq \mathcal L_{\Box \neq }$, but leave open whether $\mathcal L_{\boxdot\neq}< \mathcal L_{\Box \neq }$.
One strategy that comes to mind is to show that the class of spaces validating $\rm Kur$ is not definable in $\mathcal L_{\boxdot \neq }$.
However, as we will see, this idea will not quite work.

\section{Definability of the Kuratowski formula}\label{secDefKur}

In this section we exhibit a formula $\varkur$ which defines the same class of topological spaces as $\rm Kur$, thus showing that $\rm Kur$ itself does not suffice to show that $\mathcal L_{\boxdot\neq}< \mathcal L_{\Box \neq }$.
To this end, we introduce the abbreviation $[\varphi]\psi:= \boxdot(\varphi\rightarrow \psi)$, and set
\[\varkur : = \big ( \neg q \wedge [\neq] q \wedge  [q](\boxdot p \vee\boxdot\neg p) \big ) \rightarrow [q] p\vee [q] \neg p.\]

\begin{theorem}
For every topological space $X$, $X\models {\rm Kur}$ if and only if $X\models\varkur$.
\end{theorem}

\begin{proof}
First assume that $X\models {\rm Kur}$, and let $\val\cdot$ be any valuation on $X$.
Let $x\in X$ and assume that $x\in \val{\neg q \wedge [\neq] q \wedge  [q](\boxdot p \vee\boxdot\neg p) }$.
We claim that $x\in \val{\Box(\boxdot p\vee\boxdot \neg p) }$.
To see this, first note that from $x\in \val {\neg q \wedge[\neq] q}$ we obtain that $X\setminus \{x\} =  \val{q}$.
From $x\in \val {[q](\boxdot p \vee\boxdot\neg p)}$ we obtain a neighborhood $U$ of $x$ such that $ U\setminus \{x\}  \subseteq \val {q\rightarrow (\boxdot p \vee\boxdot\neg p) }$.
But $U\setminus \{x\} \subseteq \val q$, so $U\setminus \{x\} \subseteq \val { \boxdot p \vee\boxdot\neg p }$, witnessing that $x\in \val{\Box (\boxdot p \vee\boxdot\neg p )}$.
From $X\models{\rm Kur}$ we see that $x\in \val{ \Box p\vee \Box \neg p}$.
Assume that $x \in \val{\Box p}$.
Letting $V $ be a neighborhood of $x$ such that $V\setminus \{x\} \subseteq \val{p}$, from $X\setminus \{x \} = \val q$ we readily see that $V$ witnesses $x\in \val{[q] p}$, and thus $x\in \val{\varkur}$.
The case where $x \in \val{\Box \neg p}$ is symmetric, except that we obtain $x\in \val{[q]\neg p}$.
Since $x$ and $\val\cdot$ were arbitrary, $X\models\varkur$.

Conversely, assume that $X\models\varkur$.
Let $\val\cdot$ be any valuation on $X$, and assume that $x\in \val {\Box (\boxdot p\vee\boxdot \neg p)}$.
Let $\val\cdot'$ be the valuation identical to $\val\cdot$ except that $\val q' = X\setminus \{x\}$.
It follows that $x\in \val{\neg q \wedge [\neq] q   }'$, and any neighborhood $U$ witnessing $x\in\val{ \Box(\boxdot p \vee\boxdot \neg p)}$ also witnesses $x\in \val{  [q](\boxdot p \vee\boxdot\neg p) }$.
From $X\models\varkur$ we obtain $x\in \val {[q]p\vee [q]\neg p}'$, which by our choice of $\val q'$ readily implies $x\in \val {\Box p\vee \Box\neg p}'$, hence $x\in \val {\Box p\vee \Box\neg p} $ and thus $x\in \val {\rm Kur}
$.
Since $x$ was arbitrary, $X\models\rm Kur$.
\end{proof}

Nevertheless, a mild variant of ${\rm Kur}$ will useful for showing that $\mathcal L_\Box\not\leq \mathcal L_{\boxdot\forall}$.
For this, we first need to exhibit a class of morphisms which preserve formulas with $[\neq]$.

\section{Morphisms for the difference modality}\label{secMorph}

The semantic clauses for $\neq$ are preserved by bijections, in the following sense.
Suppose that $f\colon X\to Y$ is a bijection and $p$ is a variable, and $\val\cdot_X$, $\val\cdot_Y$ are valuations on the respective spaces satisfying $\val p_X = f^{-1}\val p_Y$.
Then, given $x\in X$, $x\in \val {\langle\neq\rangle p}_X$ iff $f(x) \in \val {\langle\neq\rangle p}_Y$.
In general, if $f$ fails to be surjective it is easy to find counterexamples to the latter equivalence.
However, if $f$ is surjective but fails to be injective, the only way to have $x \in \val {\langle\neq\rangle p}_X$ but $f(x)\not \in \val {\langle\neq\rangle p}_Y$ is if $f(x)$ is the {\em only} point of $Y$ satisfying $p$ and, moreover, there is $x'\neq x$ such that $f(x') = f(x)$.
In this case we say that $f(x)$ is $p$-unique (and $x$ is not).
It suffices for $f$ to be injective with respect to unique points for it to preserve the semantic conditions for $\langle\neq\rangle$ (and hence $[\neq]$): below, we make this precise.

\begin{definition}
Let $X$, $Y$ be topological spaces and $U\subseteq Y$.
Say that $f\colon X\to Y$ is {\em $U$-injective} if for each $u\in U$, $f^{-1}(u)$ is a singleton.
If $f$ is interior, surjective and $U$-injective, we say that $f$ is a {\em $U$-morphism.}

Let $  (Y,\val\cdot_ Y)$ be a topo-model, $\Sigma$ a set of formulas.
For a formula $\varphi$, say that $y\in Y$ is {\em $\varphi$-unique} if $\val\varphi = \{y\}$, and {\em $\Sigma$-unique} if it is $\varphi$-unique for some $\varphi\in \Sigma$.
Let ${\rm U}(\Sigma)$ be the set of $\Sigma$-unique points.
We define a {\em $\Sigma$-morphism} to be a ${\rm U}(\Sigma)$-morphism.
\end{definition}

Recall that any function $f\colon X\to Y$ defines a valuation $\val\cdot_X$ on $X$ by setting $\val p_X = f^{-1} \val p_Y$ and extending recursively to complex formulas.

\begin{lemma}
If $\Sigma \subseteq \mathcal L_{\boxdot\neq}$ is closed under subformulas and single negations and $f\colon X\to Y$ is a $\Sigma$-morphism, then for every $\varphi\in 
\Sigma$, $\val\varphi_ X = f^{-1}\val\varphi_ Y$.
\end{lemma}

\begin{proof}
Induction on formulas, with only the case for $[{\neq}]\varphi$ being non-standard.
If $f(x) \not \in \val{[{\neq}]\varphi}_ Y$, there is $y'\neq f(x)$ such that $y'\not \in \val \varphi_ Y$.
Since $f$ is surjective, $y'=f(x')$ for some $x'\in X$, which since $f$ is a function satisfies $x'\neq x$.
By the IH $x'\not \in \val\varphi_ X$, so $ x\not \in \val{[{\neq}]\varphi}_ X$.

For the other direction, if $f(x)  \in \val{[{\neq}]\varphi}_ Y$, let $x' \not \in \val\varphi_ X $; we must prove that $x ' = x$ to conclude that $x \in \val{[\neq]\varphi}_ X$.
Note that if $y  \in \val{\neg\varphi}_ Y $ it follows that $f(x) = y$, given that $f(x)  \in \val{[{\neq}]\varphi}_ Y$.
By the induction hypothesis, $f(x')\not  \in \val{ \varphi}_ Y$, which by the above yields $f(x) = f(x')\in \val{\neg\varphi}_Y$.
Thus $\val{\neg\varphi}_Y=\{f(x)\}$, and $f(x)$ is $\neg \varphi$-unique.
But, $f $ is $\Sigma$-injective, so $f(x) = f(x')$ yields $x = x'$, as required.
\end{proof}

Note that if $\Sigma$ is finite, there can be only finitely many $\Sigma$-unique points.
This will allow us to give a criterion for when, given {\em any} valuation $\val\cdot_Y$ on $Y$, there is a $\Sigma$-morphism $f\colon X\to Y$.
Below, we write $A \subseteq_{\rm fin} B$ to indicate that $A$ is a finite subset of $B$.

\begin{definition}
If $X,Y$ are topological spaces, we write $X\neqrel  Y$ if for every $U \subseteq_{\rm fin} Y$ there is a surjective, $U$-injective interior map $f\colon X\to Y$.
\end{definition}

It is easy to check that if $\Sigma$ is finite and $X\neqrel  Y$, then for any valuation on $Y$ there is a $\Sigma$-morphism $f\colon X\to Y$.
It follows that if $X\neqrel  Y$, then every $\mathcal L_{\boxdot\neq}$-formula valid on $X$ is valid on $Y$.

\begin{lemma}\label{LemmXggR}
Define
\[ \mathbb X = \{(x,y) \in \mathbb R^2 : |y| \leq |\sin(x)|\}\]
(see Figure \ref{FigX}).
Then, every formula of $\mathcal L_{\boxdot\neq}$ valid on $\mathbb X$ is valid on $\mathbb R$.
\end{lemma}

\begin{proof}
It suffices to check that $\mathbb X\neqrel  \mathbb R $.
Let $U\subseteq_{\rm fin} \mathbb R$.
Without loss of generality, we may assume that $U$ consists of multiples of $\pi$; otherwise, apply a homeomorphism $g$ to $\mathbb R$ so that $g(U)$ consists of multiples of $\pi$, which is possible since $U$ is finite.
Let $f\colon \mathbb X\to \mathbb R$ be given by $f(x,y) = x$.
Then $f$ is clearly a surjective interior map, and it is $U$-injective because if $x \in U$ then $x$ is a multiple of $\pi$, so $\sin (x) = 0$ and $f^{-1} (x) = \{(x,0)\}$.
It follows that any $\mathcal L_{\Box\neq}$-formula valid on $\mathbb X$ is valid on $\mathbb R$.
\end{proof}

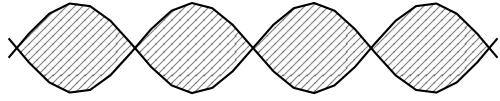
\begin{figure}

\begin{center}

  \begin{tikzpicture}[xscale=0.5,yscale=0.6]

      
 
        \fill [pattern=north east lines, opacity=.7, domain=-2*pi:2*pi, variable=\x]
         plot ({\x}, {-sin(deg(\x))})
      -- plot ({\x}, {sin(deg(\x))})
     -- cycle;



    \draw [thick,domain=-6.5:6.5, variable=\x]
      plot ({\x}, {sin(deg(\x))});  

    \draw [thick,domain=-6.5:6.5, variable=\x]
      plot ({\x}, {-sin(deg(\x))});  

  \end{tikzpicture}

\end{center}  
 
\caption{The space $\mathbb X$ looks like the shadow of an infinite braid.}
\label{FigX} 
  
\end{figure}

\begin{remark}
Unlike other notions of topological morphisms, the relation $X\neqrel  Y$ is not witnessed by a single map, but rather, the existence of a suitable map $f_U$ for each finite $U\subseteq Y$.
However, it is possible to gather these maps into a single object $(f_U)_{U\subseteq_{\rm fin} Y}$, and view the latter as a form of topological morphism tailored for logics with $\neq$.
We will call such collections of functions {\em $\neq$-morphisms.}
\end{remark}

\section{A $\mathcal L_{\boxdot\neq}$-undefinable property}\label{sectIndef}

Recall that
$ {\rm Kur} =    \Box(\boxdot p\vee  \boxdot \neg p ) \to (\Box p \vee \Box \neg p) $
is the Kuratowski formula \cite{Kuratowski22}, and is valid on locally $1$-component spaces by Lemma \ref{lemKurAx}.
Note that ${\rm Kur}$ is expressible in $\mathcal L_{\Box}$.
Our space $\mathbb X$ does not validate $\rm Kur$, but it {\em does} validate an approximate version.

\begin{lemma}
$\Box{\rm Kur}$ is valid on $\mathbb X$ but not on $\mathbb R$.
\end{lemma}

\begin{proof}
Let $\val\cdot_\mathbb X$ be any valuation on $\mathbb X$.
To see that it validates $\Box{\rm Kur}$, note that if $(x,y) \in \mathbb X$, then any small-enough neighbourhood $N$ of $(x,y)$ has the property that if $(x', y') \in N\setminus \{(x,y)\}$, then $x'$ is not a multiple of $\pi$.
But it is easy to see that any small-enough punctured ball around $(x', y')$ is connected (see Figure \ref{FigX}), so $(x',y' ) \in \val{\rm Kur} _\mathbb X$, and thus $N$ witnesses that $(x,y) \in \val{\Box{\rm Kur} }_\mathbb X$.

To see that $\Box{\rm Kur}$ is not valid on $\mathbb R$, let $\val\cdot_\mathbb R$ be such that
\[\val p_\mathbb R = \{x\in\mathbb R : \exists n\geq 0 \text{ s.t. } 2^{-2n-1 } <x< 2^{-2n } \} .\]
Then, any punctured neighborhood $U$ of $0$ contains a point of the form $x =  2^{-2n-1 }$ for large-enough $n$.
But $N = (2^{-2n-2 }, 2^{-2n  })$ is a neighborhood of $x$, and all points in $N$ to the right of $x$ satisfy $p$, hence $\boxdot p$, and all points in $N$ to the left of $x$ satisfy $\neg p$, hence $\boxdot\neg p$.
Thus $x\in \val{\Box(\boxdot p \vee \boxdot \neg p)}_\mathbb R$.
However, neither $x\in \val{\Box p}_\mathbb R$ nor $x\in \val{\Box \neg p}_\mathbb R$, so $x\notin \val{\Box {\rm Kur}}_\mathbb R$.
Since $U$ was arbitrary, it follows that $0\notin \val{ \Box {\rm Kur } }_\mathbb R$.
\end{proof}

However, by Lemma \ref{LemmXggR}, we have that any $\mathcal L_{\boxdot\neq}$-formula valid on $\mathbb X$ is also valid on $\mathbb R$, so no $\mathcal L_{\boxdot\neq}$-formula defines the same class of spaces as $ \Box {\rm Kur }$.
We thus obtain the following.

\begin{theorem}
The language $\mathcal L_\Box$ is not reducible to $\mathcal L_{\boxdot\neq}$.
\end{theorem}

As a corollary, we obtain that $\mathcal L_{\boxdot\neq}<\mathcal L_{\Box\neq}$.

\section{Concluding remarks}

We have closed the most prominent question left open by \cite{KS} in the comparison of topological modal languages by establishing that $\mathcal L_{\boxdot\neq}<\mathcal L_{\Box\neq}$: in fact, we showed that $\mathcal L_{\Box}\not\leq \mathcal L_{\boxdot\neq}$.
However, our main contribution is arguably a notion of topological morphism which preserves formulas with the difference modality.
Unlike related morphisms for similar languages, preservation of formulas with $[\neq]$ requires not only one map, but a family of maps indexed by the finite subsets of the codomain.

One interesting line of inquiry opened by $\neq$-morphisms is that of {\em succinctness:} Fern\'andez-Duque and Iliev~\cite{FernandezSuccinctness} show that despite being less expressive, $\mathcal L_\boxdot$ is more succinct than $\mathcal L_\Box$ for certain formulas.
Similar results could hold for $\mathcal L_{\boxdot\neq}$ vs.~$\mathcal L_{\Box\neq}$, and would require a refined notion of $\neq$-morphisms.

Finally, we remark that interior maps do not always preserve formulas with $\Box$: more restrictive maps, sometimes called {\em $\bm d$-morphisms,} are needed.
By combining {$\bm d$-morphisms} with $\neq$-morphisms, it would be possible to exhibit classes of spaces which are {\em not} distinguished by $\mathcal L_{\Box\neq}$, thus establishing non-trivial limitations for the expressive power of the full topological modal language.

\section*{Acknowledgements}

I would like to thank Alexandru Baltag and Nick Bezhanishvili for bringing my attention to the difference modality in topological spaces and for some interesting discussions.

\bibliographystyle{plain}
\bibliography{biblio}

\end{document}